\documentclass[12pt]{amsart}
\usepackage{amscd, amssymb,latexsym,amsmath, amscd, amsmath}
\usepackage{amsthm}
\usepackage{amstext}
\usepackage{anysize}
\usepackage[sc]{mathpazo} 
\usepackage{enumitem}
\usepackage{pst-node}
\marginsize{2.5cm}{2.5cm}{2.5cm}{2.5cm}
\usepackage{graphics}
\usepackage{color}
\usepackage{tikz}
\usepackage{tikz-cd}

\DeclareFontFamily{U}{wncy}{}
    \DeclareFontShape{U}{wncy}{m}{n}{<->wncyr10}{}
    \DeclareSymbolFont{mcy}{U}{wncy}{m}{n}
    \DeclareMathSymbol{\Sh}{\mathord}{mcy}{"58}

\begin{document}
\title{A mod-$p$ Artin-Tate conjecture, and generalized  Herbrand-Ribet}
\author{Dipendra Prasad}
\maketitle
{\hfill \today}


\theoremstyle{plain}

\newtheorem{theorem}{Theorem}
\newtheorem{proposition}{Proposition}
\newtheorem*{thm}{Theorem}
\newtheorem{cor}{Corollary}
\newtheorem{lemma}{Lemma}

\newtheorem{conjecture}[theorem]{Conjecture}


\theoremstyle{definition}
\newtheorem{remark}{Remark}
\newtheorem{definition}{Definition}
\newtheorem{question}{Question}

\def\A{{\mathbb A}}
\def\C{{\mathbb C}}                               
\def\F{{\mathbb F}}                               
\def\Ga{{\mathbb G}_a}                            
\def\Gm{{\mathbb G}_m}                            
\def\H{{\mathbb H}}                               
\def\K{{\mathbb K}}
\def\L{\mathcal L}
\def\M{{\mathbb M}}                               
\def\N{{\mathbb N}}                               
\def\O{{\mathcal O}}
\def\Q{{\mathbb Q}}                               
\def\R{{\mathbb R}}                               
\def\Z{{\mathbb Z}}                               
\def\PP{{\mathbb P}}                               

\def\k{{\overline{k}}}                            
\def\m{\mathfrak{m}}
\def\p{\mathfrak{p}}
\def\wa{\mathfrak{a}}
\def\wc{\mathfrak{c}}
\def\wf{\mathfrak{f}}
\def\Aut{\operatorname{Aut}}                      
\def\End{{\operatorname{End}}}                    
\def\GL{\operatorname{GL}}                        
\def\GSO{\operatorname{GSO}}                        
\def\GSp{\operatorname{GSp}}                        
\def\Gal{\operatorname{Gal}}                      
\def\Hom{\operatorname{Hom}}                      
\def\Ind{\operatorname{Ind}}
\def\ind{\operatorname{ind}}

\def\P{{\mathbb{P}}}
\def\PSL{\operatorname{PSL}}                      
\def\PSO{\operatorname{PSO}}                      
\def\PSU{\operatorname{PSU}}                      
\def\PGL{\operatorname{PGL}}                      
\def\PSp{\operatorname{PSp}}                      
\def\Rad{\mathfrak{Rad}}
\def\Rep{\mathfrak{Rep}}
\def\Res{\operatorname{Res}}
\def\SL{\operatorname{SL}}           
\def\alg{\operatorname{alg}}           
\def\SL{\operatorname{SL}}                        
\def\SO{\operatorname{SO}}                        
\def\SU{\operatorname{SU}}                        
\def\Sp{\operatorname{Sp}}                        
\def\Spin{\operatorname{Spin}}                    
\def\Sym{\operatorname{Sym}}
\def\Sel{\operatorname{Sel}}
\def\Or{\operatorname{O}}                         
\def\Un{\operatorname{U}}                         
\def\Ps{\operatorname{Ps}}                         

\def\Fr{\operatorname{Fr}}                         
\def\det{\operatorname{det}}                      
\def\tr{\operatorname{tr}}                        
\def\into{{~\hookrightarrow~}}                    
\def\iso{{\stackrel{\sim}{~\longrightarrow~}}}    
\def\lra{{~\longrightarrow~}}                     %
\def\onto{{~\twoheadrightarrow~}}                 
\def\op{{\oplus}}                                 
\def\ot{{\otimes}}                                
\def\rad{\mathfrak{rad}}                          %
\def\sdp{{~\mathbin{{\triangleright}\!{<}}}}      
\def\siso{{\stackrel{\sim}{~\rightarrow~}}}       

\def\v2{{\vskip2mm}}

\begin{abstract}
Following the natural instinct that when a group operates on a number field then every term
in the class number formula should factorize `compatibly' according to the
representation theory (both complex and modular) of the group, we are led ---
in the spirit of Herbrand-Ribet's theorem on the $p$-component of the class number of $\Q(\zeta_p)$ --- 
to some natural questions
about the $p$-part of the classgroup of any CM Galois extension of $\Q$ as a module for 
$\Gal(K/\Q)$.
The compatible factorization of the class number formula 
is at the basis of {\it Stark's conjecture}, where one is mostly interested in factorizing the regulator term --- whereas 
for us in this paper, we put ourselves in a situation where the regulator term can be ignored, and it is the factorization   
of the classnumber that we seek. 
All this is  presumably part of various `equivariant' conjectures in arithmetic-geometry, such 
as `equivariant Tamagawa number conjecture',  but the literature does not seem to address this question in any  precise 
way. In trying to formulate these questions, we are naturally led to consider $L(0,\rho)$, for $\rho$ an Artin representation, 
 in situations where this is known to be nonzero and algebraic, and it is important for us 
to understand if this is $\p$-integral for a 
prime $\p$ of the ring of algebraic integers $\bar{\Z}$ in $\C$, that we call {\it mod-$p$ Artin-Tate conjecture}.
As an attentive reader will notice, the most innocuous term in the class number formula, 
the number of roots of unity, plays an important role for us --- it being the only term in the denominator, is responsible for all the poles!  

\end{abstract}
\tableofcontents

\section{Introduction}

Let $F$ be a number field contained in $\C$ with $\bar{\Q}$ its algebraic closure in $\C$. 
Let $\rho: \Gal(\bar{\Q}/F) \rightarrow \GL_n(\C)$ 
be an irreducible Galois representation with $L(s,\rho)$ 
  its associated Artin $L$-function. According to a famous conjecture of Artin, 
$L(s, \rho)$ has an analytic continuation to 
 an entire function on $\C$ unless $\rho$ is the trivial representation, in which case it has a 
unique pole at $s=1$ which is simple.

More generally, let $M$ be an irreducible  motive over $\Q$ with $L(s,M)$ its associated $L$-function. According to Tate,
$L(s,M)$  has  an analytic continuation to 
 an entire function on $\C$ unless $M$ is a twisted Tate motive $\Q[j]$ with $\Q[1]$ the motive associated to $\Gm$. 
For the motive $\Q = \Q[0]$, $L(s, \Q) = \zeta_\Q(s)$, the usual Riemann zeta function, which has a 
unique pole at $s=1$ which is simple.

This paper will deal with certain Artin representations $\rho: \Gal(\bar{\Q}/F) \rightarrow \GL_n(\C)$ for which we will know a priori that $L(0, \rho)$ is a nonzero
algebraic number (in particular, $F$ will be totally real). It is then an important question to understand the nature of the algebraic number $L(0,\rho)$:  to know if it is an 
algebraic integer, but if not, what are its possible denominators. We think of the possible denominators 
in $ L(0,\rho),$ as existence of poles for $L(0,\rho),$ at the corresponding prime ideals of $\bar{\Z}$.  It is thus 
analogous to the conjectures of Artin and Tate, both in its aim ---  and as we will see ---  in its formulation.  
Since we have chosen to understand $L$-values at 0 instead of 1 which is where Artin and Tate conjectures are formulated, 
there is an `ugly' twist by $\omega_p$ --- the action of $\Gal(\bar{\Q}/\Q)$ on the $p$-th roots of unity --- 
throughout the paper, giving a natural character  $\omega_p: \Gal(\bar{\Q}/\Q) \rightarrow (\Z/p)^\times$, also a character of 
$\Gal(\bar{\Q}/L)$ for $L$ any algebraic extension of $\Q$, 
as well as a character of $\Gal(L/\Q)$ 
if $L$ is a Galois  extension of $\Q$ containing $p$-th roots of unity; if there are no non-trivial
$p$-th roots of unity in $L$, we will define $\omega_p$ to be the trivial character of $\Gal(L/\Q)$.

We now
fix some notation. We will fix an isomorphism of $\bar{\Bbb Q}_p$ 
with $\C$ where $\bar{\Bbb Q}_p$ is
 a fixed algebraic closure of ${\Bbb Q}_p$, the field
of $p$-adic numbers. This allows one to define $\p$, a prime ideal in  $\bar{\Bbb Z}$,  the integral closure of
${\Bbb Z}$ in ${\Bbb C}$, over the prime ideal generated by $p$ in $\Z$. The prime $p$ will always be an odd prime in this paper.

All the finite dimensional representations of finite 
groups in this paper will take values 
in $\GL_n(\bar{\Bbb Q}_p)$,
and therefore in $\GL_n({\C})$, as well as $\GL_n(\bar{\Bbb Z}_p)$.
It thus makes sense to talk of `reduction modulo $\p$' of (complex)
representations of finite groups. These reduced representations are
well defined up to semi-simplification on vector spaces over
$\bar{\Bbb F}_p$ (theorem of Brauer-Nesbitt); we denote the reduction modulo
$\p$ of representations as $\rho \rightarrow \bar{\rho}$.

If $F$ is a finite Galois extension of ${\Bbb Q}$ with Galois group $G$, 
then it
is well-known that 
the zeta function $\zeta_F(s)$ can be factorized as
$$\zeta_F(s) = \prod_{\rho} L(s,\rho)^{\dim \rho},$$
where $\rho$ ranges over all the irreducible complex representations of $G$, 
and $L(s,\rho)$ denotes the Artin $L$-function associated to $\rho$.

According to the class number formula, we have,
$$\zeta_F(s) = -\frac{hR}{w}s^{r_1+r_2-1}+ 
{\rm
~~higher~order~terms},$$
where $r_1,r_2, h, R, w$ are the standard invariants associated to $F$: 
$r_1$, the number of real embeddings; $r_2$, number of pairs of 
complex conjugate embeddings which are not real;
$h$, the class number of $F$; $R$, the regulator, and $w$ the number of 
roots of unity in $F$. 

This paper considers $\zeta_E/\zeta_F$ where $E$ is a CM field with $F$ its totally real subfield, in which case 
$r_1+r_2$ is the same for $E$ as for $F$, and the regulators of $E$ and $F$ too are the same except for a possible power of 2.
Therefore for $\tau$ the complex conjugation on $\C$,
$$\zeta_E/\zeta_F(0)= \prod_{\rho(\tau)=-1} L(0,\rho)^{\dim \rho} = \frac{h_E/h_F}{w_E/w_F},$$
where each of the $L$-values $L(0,\rho)$ in the above expression 
are nonzero algebraic numbers by a theorem of Klingen and Siegel.

In this identity, observe that $L$-functions are associated to $\C$-representations of $\Gal(E/\Q)$, whereas the classgroups
of $E$ and $F$ are finite Galois modules.  Modulo some details, we basically assert that for each odd prime $p$,
each irreducible $\C$-representation $\rho$  of $\Gal(E/\Q)$ contributes a certain number of copies (depending on $\p$-adic valuation of $L(0,\rho)$) of 
$\bar{\rho}$ to the classgroup of $E$ tensored with $\F_p$ modulo the classgroup of $F$ tensored with $\F_p$ (up to semi-simplification).  
 This is exactly what happens for $E=\Q(\zeta_p)$ by the theorems of Herbrand and Ribet which is one of the main
motivating example for all that we do here, and this is what we will review next.
\section{The Herbrand-Ribet theorem}
In this section we recall the Herbrand-Ribet theorem from the point of view of this paper. 
We refer to \cite{Ri1} for the original work of Ribet, and \cite{Was} 
for an exposition on the theorem together with a proof of Herbrand's theorem. 

There are actually 
two a priori important aspects of  the Herbrand-Ribet theorem  dealing with the $p$-component of the classgroup for $\Q(\zeta_p)$. First,  the 
Galois group $\Gal(\Q(\zeta_p)/\Q) = (\Z/p)^\times$, being
a cyclic group of order $(p-1)$, its action on the $p$-component of the classgroup is semi-simple,
and the $p$-component of the classgroup 
can be written as a direct sum of eigenspaces for $(\Z/p)^\times$. We do not consider this aspect of 
Herbrand-Ribet theorem to be important, and simply consider semi-simplification of representations 
of Galois group on classgroups to be a good enough substitute.

The second --- and more serious --- aspect of Herbrand-Ribet theorem is that among the characters of 
$\Gal(\Q(\zeta_p)/\Q) = (\Z/p)^\times$, only the odd characters, i.e.,
characters $\chi: (\Z/p)^\times \rightarrow \Q_p^\times$ 
with $\chi(-1)=-1$ present themselves --- as it is only for these that there is any result about the $\chi$-eigencomponent 
in the classgroup, and even among these, the Teichm\"uller character
$\omega_p: (\Z/p)^\times \rightarrow \Q_p^\times$ plays a role different from other characters of $(\Z/p)^\times$.
(Note that earlier we have used $\omega_p$ for the action of $\Gal(\bar{\Q}/\Q)$ on the $p$-th roots of unity, 
giving a natural character  $\omega_p: \Gal(\bar{\Q}/\Q) \rightarrow (\Z/p)^\times$, 
as well as to its restriction to  
$\Gal(\bar{\Q}/L)$ for $L$ any algebraic extension of $\Q$. Since $\Gal({\Q(\zeta_p)}/\Q)$ is canonically 
isomorphic to $ (\Z/p)^\times$, the two roles that $\omega_p$ will play throughout the paper are actually the same.)

To elaborate on this aspect (the role of  `odd' characters in Herbrand-Ribet theorem), 
observe that the class number formula
$$\zeta_F(s) = -\frac{hR}{w}s^{r_1+r_2-1}+ 
{\rm
~~higher~order~terms},$$  
can be considered both for $F = \Q(\zeta_p)$ 
as well as its maximal real subfield $F^+ = \Q(\zeta_p)^+$. It is known that, cf. Prop 4.16 in \cite{Was}, 
$$R/R^+ = 2^{\frac{p-3}{2}},$$
where $R$ is the regulator for $\Q(\zeta_p)$ and $R^+$ is the regulator for $\Q(\zeta_p)^+$. We will similarly denote
$h$ and $h^+$ to be the order of the two class groups, with $h^- = h/h^+$, an integer.
 
Dividing the class number formula of $\Q(\zeta_p)$ by that of  $\Q(\zeta_p)^+$,
we find,
$$\prod_{\chi {\rm ~an ~odd~ character~ of~ } (\Z/p)^\times }L(0, \chi) = \frac{1}{p} \cdot \frac{h}{h^+} \cdot 2^{\frac{p-3}{2}}
,\quad \quad \quad \quad \hfill(*)$$
the factor $1/p$ arising because there are $2p$ roots of unity in $\Q(\zeta_p)$ and only $2$
in $\Q(\zeta_p)^+$.

It is known that for $\chi$ an odd character 
of $ (\Z/p)^\times $, $L(0, \chi) $ is an algebraic number which 
is given in terms of the generalized Bernoulli number $B_{1,\chi}$ as follows: $$L(0, \chi) = -B_{1,\chi} = - \frac{1}{p}\sum_{a=1}^{a=p} a \chi(a).$$

It is easy to see that $pB_{1,\omega_p^{p-2}} \equiv (p-1) \bmod p$ since $a\omega_p^{p-2}(a)$ is the trivial character of 
$(\Z/p)^\times$ whereas for all the other characters of $ (\Z/p)^\times $, $L(0, \chi) $ is not only an algebraic 
number but is $p$-adic integral (Schur orthogonality!); all this is clear by looking at the expression:
 $$L(0, \chi) = -B_{1,\chi} = - \frac{1}{p}\sum_{a=1}^{a=p} a \chi(a).$$

Rewrite the equation $(*)$ up to $p$-adic units
as,
$$\prod_{\begin{array}{cc} \chi {\rm ~an ~odd~ character~ of~ } (\Z/p)^\times \\ \chi \not = \omega_p^{p-2} = \omega_p^{-1}\end{array}}L(0, \chi) = \frac{h}{h^+},$$
where we note that both sides of the equality are $p$-adic integral elements; in fact, since all 
characters $\chi: (\Z/p)^\times \rightarrow \bar{\Q}_p^\times$ take values in $\Z_p$, for $\chi \not = \omega_p^{-1}$, 
$L(0,\chi) \in \Z_p$. 
This when interpreted --- just an interpretation in the optic of this paper without any suggestions for proof in either direction! --- for each $\chi$ component on the two sides of this equality
amounts to the theorem of Herbrand and  Ribet which asserts that  $p$ divides  $L(0,\chi)=-B_{1,\chi}$ 
for $\chi 
$ an odd character of $ (\Z/p)^\times $, which is not $\omega_p^{p-2}$,
if and only if  the corresponding 
$\chi^{-1}$-eigencomponent of the classgroup of $\Q(\zeta_p)$ is nontrivial. (Note the $\chi^{-1}$, and not $\chi$!)
Furthermore, the character $\omega_p$ does not appear in the $p$-classgroup of $\Q(\zeta_p)$. It can happen that $L(0,\chi)$ 
is divisible by higher powers of $p$ than 1, and one expects --- this is not proven yet! --- that in such cases, the corresponding
$\chi^{-1}$-eigencomponent of the classgroup of $\Q(\zeta_p)$ is $\Z/p^{({\rm val}_pL(0,\chi))}$, in particular, it still 
has $p$-rank 1. (By 
Mazur-Wiles \cite{Ma-Wi}, $\chi^{-1}$-eigencomponent of the classgroup of $\Q(\zeta_p)$ is of order $p^{({\rm val}_pL(0,\chi))}$.) 
 
The work of Ribet was to prove that if 
$p| B_{1,\chi}$, $\chi^{-1}$-eigencomponent of the classgroup of $\Q(\zeta_p)$ is nontrivial by constructing an unramified
extension of $\Q(\zeta_p)$ by using a congruence between a holomorphic cusp form and an Eisenstein series on $\GL_2(\A_\Q)$.

To be able to use the class number formula in other situations, we will 
need to have the integrality of $L(0,\chi)$ for $\chi$ a character associated to the Galois group of a number field, or even of $L(0,\rho)$ for general irreducible representations $\rho$ of the Galois group of a number field,
 in more situations that we call mod $p$ Artin-Tate conjecture. We begin with the following lemma.

\begin{lemma} \label{class}
Let $\ell, p$ be any  two odd primes ($\ell = p$ is allowed). Let $F$ be a totally real number field, and $E=F(\zeta_{p^n})$ be a quadratic CM extension of $F$. Let $h_{\ell}^-(E)$ 
be the order of the $\ell$-primary part of the classgroup of $E$ on which complex conjugation acts by $-1$; define $h_{\ell}^-(\Q(\zeta_{p^n}))
$ similarly.  Then
$$h_{\ell}^-(\Q(\zeta_{p^n}))| h_{\ell}^-(E).$$ 
\end{lemma}
  \begin{proof}By classfield theory,  it suffices to prove that a cyclic degree $\ell$  unramified extension, say $L$, of $\Q(\zeta_{p^n})$ on which complex conjugation acts by $-1$ on $\Gal(L/\Q(\zeta_{p^n})) \cong \Z/\ell$ 
when inflated to $E$ remains cyclic of degree $\ell$, i.e., 
the degree of $LE/E$ is $\ell$.  Assuming the contrary, we have $LE=E$, and since 
degree of $E/F$ is 2, and $p$ is odd, we must have $LF=F$, i.e., $L\subset F$,
therefore complex conjugation must act trivially on $L$. On the other hand, 
we know that complex  conjugation does not act trivially on $L$. \end{proof}
\begin{remark} It may be noted that we are not asserting that 
if $\Q(\zeta_{p^n}) \subset E $, $h(\Q(\zeta_{p^n}))| h(E).$ \end{remark}

Let $E$ be a CM number field which we assume is Galois over $\Q$. 
Assume that $E$ contains $p^n$-th roots of unity but no $p^{n+1}$-st root of unity. Let $F$ be the totally
real subfield of $E$ with $[E:F]=2$. Let $G={\Gal}(E/Q)$ 
with $-1 \in G$, the complex conjugation in $G$.

We have,
$$\zeta_E(s) = \prod_{\rho} L_{\Q}(s, \rho)^{\dim \rho},$$
$$\zeta_{F}(s) = \prod_{\rho(-1)=1} L_{\Q}(s, \rho)^{\dim \rho},$$
$$\zeta_E/\zeta_{F}(s) = \prod_{\rho(-1)=-1} L_{\Q}(s, \rho)^{\dim \rho},$$
where all the products above are over irreducible representations $\rho$ of $G={\Gal}(E/\Q)$.

By the class number formula,
$$h^-(E)/{p^n} = \prod_{\rho(-1)=-1} L_{\Q}(0, \rho)^{\dim \rho}\quad(\star1).$$

Similarly,
$$h^-(\Q(\zeta_{p^n}))/{p^n} =
 \prod_{\chi(-1)=-1} L_{\Q}(0, \chi)\quad(\star2).$$

Dividing the equation $(\star1)$ by $(\star2)$, we have,
 $$h^-(E)/h^-(\Q(\zeta_{p^n}))  = \prod_{
\begin{array}{c}
\rho(-1) =  -1,  \\
\rho \not = \chi    
          \end{array} } L_{\Q}(0, \rho)^{\dim \rho}\quad(\star3),$$
where  the product on the right is taken  over irreducible representations $\rho$ of $G={\Gal}(E/\Q)$ 
for which $\rho(-1)=-1$, and  which are not cyclotomic characters 
of the form $\chi: \Gal(\Q(\zeta_{p^n})/\Q)= (\Z/p^n)^{\times}  \rightarrow  \C^\times$.

It is known that $L(0,\rho) \in \bar{\Bbb Q}^\times$ for $\rho(-1) =-1$.
This is a simple consequence of a theorem due to Klingen and Siegel that partial zeta
functions of a totally real number field take rational values at all 
non-positive integers,
cf. Tate's book \cite{Tate}.  (Note that to prove  $L(0,\rho) \in \bar{\Bbb Q}^\times$ for $\rho(-1) =-1$, it suffices by
Brauer to prove it for abelian CM extensions by a Lemma of Serre 
cf. Lemma 1.3 of Chapter III of Tate's book \cite{Tate}.)

By Lemma \ref{class}, the left hand side of the equation $(\star3)$ is integral (except for powers of 2), and we would like
to suggest the same for  each term on the right hand side of the equation $(\star3)$.

The following conjecture about $L(0,\rho)$ extends the known integrality properties  of $L(0, \chi) = -B_{1,\chi} = - \frac{1}{p}\sum_{a=1}^{a=p} a \chi(a),$ 
encountered and used earlier. The formulation of the conjecture also assumes known integrality properties
about $L(0,\chi)$ for $\chi: \Gal(\Q(\zeta_n)/\Q)= (\Z/n)^\times \rightarrow \C^\times$ discussed in the last section of this paper.

\begin{conjecture} \label{artin-tate} (mod $p$ analogue of the Artin-Tate conjecture) Let $\rho$ be an irreducible representation of $\Gal(\bar{\Q}/\Q)$ cutting out a CM extension 
$E$ of $\Q$ with $\rho(-1)=-1$ where $-1$ is the complex conjugation in $\Gal(E/\Q)$. Then unless $\rho$ is a one dimensional
representation 
factoring through $\Gal(\Q(\zeta_{p^n})/\Q)$ (for some prime $p$) with $\bar{\rho}$ the reduction of $\rho$ modulo $p$ being $\bar{\rho}= \omega_p^{-1}$, $L(0,\rho) \in \bar{\Q}$ 
is integral outside 2, i.e., $L(0,\rho) \in \bar{\Z}[\frac{1}{2}].$ 

\end{conjecture}

We next recall the following theorem of
Deligne-Ribet, cf. \cite{DR}, which could be considered as a weaker version of Conjecture \ref{artin-tate}.

\begin{thm}
Let $F$ be a  totally real number field, and  let 
$\chi: \Gal(\bar{\Q}/F) \rightarrow \bar{\Q}^\times$ be a character of finite order cutting out a 
CM extension $K$ of $F$ (which is not totally real). Let $w$ be the order of the group of roots of unity in $K$. Then,
$$wL(0,\chi) \in \bar{\Z}.$$
\end{thm}

In fact Conjecture \ref{artin-tate} can be used to make precise the above  theorem of Deligne-Ribet as follows; 
the simple argument 
using the fact that the Artin $L$-function is invariant under induction from $\Gal(\bar{\Q}/F)$ to 
$\Gal(\bar{\Q}/\Q)$ will be left to the reader.

\begin{conjecture} 
Let $F$ be a totally real number field, and $\chi:\A_F^\times/F^\times \rightarrow \bar{\Z}_p^\times$ 
a finite order character, cutting out a non-real but CM extension. Then if $L_F(s,\chi) \not \in \bar{\Z}_p$,  
\begin{enumerate}
\item $\chi$ mod $p$ is $\omega_p^{-1}$.

\item $\chi$ is a character of $\A_F^\times/F^\times$ associated to a character of the Galois group
$\Gal( F(\zeta_q)/F)$ for some $q$ which is a power of $p$.
\end{enumerate}

\end{conjecture}

\begin{remark} In the examples that I know, which are for  characters $\chi: \Gal(\bar{\Q}/\Q) 
\rightarrow \bar{\Q}_p^\times$ 
with $\chi = \omega_p^{-1}$ (mod $p$), if $L(0,\chi)$ has a (mod $p$) pole, the pole is of order 1; 
more precisely, if $L= {\Q}_p[\chi(
 \Gal(\bar{\Q}_p/\Q_p) )]$ is the subfield of $\bar{\Q}_p$ generated by the image under $\chi$ of the decomposition group at $p$, then $L(0,\chi)$ is the 
inverse of a uniformizer of this field $L$. It would be nice to  know if this is the case for 
characters $\chi$ of  $\Gal(\bar{\Q}/F)$ for $F$ arbitrary.  
This would be 
in the spirit of classical Artin's conjecture where the only possible poles of $L(1,\rho)$, for $\rho$ an irreducible representation of  $\Gal(\bar{\Q}/F)$, are simple.

\end{remark}

\section{Proposed generalization of  Herbrand-Ribet for CM number fields} 
The Herbrand-Ribet theorem is about the relationship of $L$-values $L(0,\chi)$ with the $\chi^{-1}$-eigencomponent 
of the classgroup of $\Q(\zeta_p)$. In the last section, we have proposed a precise conjecture about 
integrality properties for the $L$ values $L(0,\rho)$. In this section, we now propose their relationship to classgroups.

We begin by  introducing some notation involved in  constructing in a functorial way an elementary abelian $p$-group $\overline{A}[p]$ 
out of a finite abelian group
$A$ with 
\begin{enumerate}
\item  $p \cdot \overline{A}[p] = 0$,
\item the cardinality of $\overline{A}[p]$ 
equals the cardinality of the $p$-Sylow subgroup of $A$.
\end{enumerate}
  
We define $\overline{A}[p]$ to be the direct sum of the $p$-groups: $p^iA/{p^{i+1}A}$ for $i \geq 0$.
If $A$ is $G$-module, then naturally, $\overline{A}[p]$ 
too is a $G$-mdule.  If $A$ is a $G$-module, 
then we let $\overline{A}[p]^{ss}$ be the semi-simplification of the corresponding $G$-module $\overline{A}[p]$ over $\F_p$. 

Since according to  the theorem of Klingen and Siegel, the value $L(0,\rho)$ for an odd representation $\rho$ of 
$\Gal(\bar{\Q}/k)$, where $k$ is a totally real number field, belongs to the algebraic number field 
generated by the character values of $\rho$, and since we are trying to equate powers of $p$ appearing on the two sides
of the class number formula, it will be important to consider only those representations
 $\rho: \Gal(\bar{\Q}/\Q) \rightarrow \GL_n(\bar{\Q}_p)$ which actually take values in $\Q_p^{\rm unr}$, the maximal 
unramified extension of $\Q_p$. Observe that the Brauer group of $\Q_p^{\rm unr}$ is trivial, and thus an irreducible
 representation of a finite group is defined over $\Q_p^{\rm unr}$ 
if and only if its character is defined  over $\Q_p^{\rm unr}$. If an irreducible representation $\pi$ of a finite group 
is defined over  $\bar{\Q}_p$, one can  take the sum of the Galois conjugates $\pi^\sigma$ of $\pi$ for  $\sigma \in \Gal(\bar{\Q}_p/ \Q_p^{\rm unr})$,
to construct canonically an irreducible representation, say $\langle \pi \rangle$ 
over $\Q_p^{\rm unr}$. The representation
$\pi$ can be reduced modulo $\p$ and the representation $\langle \pi \rangle$ 
modulo $p$, and the semi-simplification of these reductions are related by
$$ \overline {\langle \pi \rangle}^{ss} \cong d \overline { \pi }^{ss},$$
where $d$ is the number of distinct Galois conjugates of $\pi$ under  $\Gal(\bar{\Q}_p/ \Q_p^{\rm unr})$.

Let $E$ be a Galois CM extension of a totally real number field $k$ with $F$ the totally
real subfield of $E$ with $[E:F] =2$, and $ G = \Gal(E/k)$.
Let $\tau$ denote the element of order 2 in the 
Galois group of $E$ over $F$. 

Let $H_E$ (resp. $H_{F}$) denote the class group of $E$ (resp. $F$).
Observe that the kernel of the natural map from $H_{F}$ to $H_E$ 
is a 2-group. (This follows from using the norm mapping from
$H_E$ to $H_{F}$.) Therefore since we are interested in $p$-primary
components for only odd primes $p$, $H_{F}$ can be considered to be
a subgroup of $H_E$, and the quotient $H_E/H_{F}$ becomes a $G$-module
of order $h_E/h_{F}$.

The following conjecture on the structure of the {\it minus-part} of the classgroup of $E$ (as a module for the Galois group $G$) 
is arrived at by considering the 
$p$-adic valuations of the two sides of the class number formula:
$$\zeta_E/\zeta_{F}(0)  = \prod_{\rho(\tau)=-1} L(0,\rho)^{\dim \rho} = \prod_{\rho(\tau)=-1} L(0,\langle \rho \rangle )^{\dim \rho} = 
\frac{h_E}{h_{F}}
\frac{1}{w_E},$$
with $E,F,k$ as above, and  the first product  taken 
over all irreducible representations $\rho$ of $G=\Gal(E/k)$ with values in
$\GL_n(\bar{\Q}_p) $, whereas the second one is over all irreducible representations $ \langle \rho \rangle $ of $G=\Gal(E/k)$ with values in
$\GL_n({\Q}_p^{\rm unr}) $. Since we are formulating the conjecture below based on equality of ($p$-adic valuations of)
  numbers in the class number formula, 
it is not sensitive to the subtlety  discussed earlier about 
$\chi$-eigencomponent in the classgroup of $\Q(\mu_p)$ being cyclic or not; all we care is their order.

\begin{conjecture} \label{classgroups}
Let  $E$ be a CM, Galois  extension of  a totally real number field $k$,  
with $F$ the totally real subfield of $E$, and $\tau \in \Gal(E/F)$, the nontrivial element of the Galois group. Let  
$\langle \rho \rangle : \Gal(E/k) \rightarrow \GL_{dn}({\Q}_p^{\rm unr}) $ be an irreducible, odd (i.e., $\rho(\tau) = -1$) 
representation of $\Gal(E/k)$ associated to an irreducible representation  
$\rho : \Gal(E/k) \rightarrow \GL_n(\bar{\Q}_p) $ as above,
with $\bar{\rho}$ the semi-simplification of the  reduction of $\rho$ mod $\p$ for $p$ 
an odd prime. Let $\omega_p: \Gal(E/k) \rightarrow (\Z/p)^\times$ be the action of $\Gal(E/k)$ on 
the $p$-th roots of unity in $E$ (so $\omega_p = 1$ if $\zeta_p \not \in E$). 
Then if $\bar{\rho} \not =  \omega_p^{-1}$ (for example if $\dim \rho > 1$), the representation $\langle \rho \rangle $ contributes 
 $v_p(L(0,\langle \rho \rangle ))$ many copies of $\bar{\rho}$ to  $\overline{H_E/H_{F}} [p]^{ss}$, 
with different $\langle \rho \rangle $'s contributing independently to $\overline{H_E/H_{F}} [p]^{ss}$, 
filling it up
except for the $\omega_p$-component. If $\omega_p \not = 1$, we make no assertion on the $\omega_p$-component in 
$\overline{H_E/H_{F}} [p]^{ss}$, 
but if $\omega_p =1$,
 there is no  $\omega_p$-component inside 
$\overline{H_E/H_{F}} [p]^{ss}$.

\end{conjecture}

\begin{remark} I should add that Ribet's theorem is specific to $\Q(\zeta_p)$ and although this section is
very general, it could also be specialized to a CM abelian extension $E$ of $\Q$, and the action of
the Galois group $\Gal(E/\Q)$ on the full class group of $E$. Since class group of an abelian extension 
is not totally obvious from the classgroup of the corresponding cyclotomic field $\Q(\zeta_n)$, 
even if we knew everything in the style of Ribet for $\Q(\zeta_n)$, presumably there is still some work left to be done, and not just book keeping (for $n$ which is a composite number)!
\end{remark}

\section{Integrality of  Abelian $L$-values for $\Q$}

The aim of this section is to prove certain 
results on integrality of $L(0,\chi)$ for $\chi$ an odd Dirichlet character of $\Q$ which go as first examples of all the integrality conjectures made in this paper. Although these are all well-known results, we have decided to give our proofs.

\begin{lemma} \label{simple}
For integers $m>1,n>1$, with $(m,n)=1$,
let $\chi = \chi_1 \times \chi_2$ be a primitive Dirichlet character on $(\Z/mn\Z)^\times = (\Z/m\Z)^\times \times  (\Z/n\Z)^\times$ with $\chi(-1)=-1$.  Then,
$$L(0,\chi) = -B_{1,\chi} = -\frac{1}{mn}\sum_{a=1}^{mn} a \chi(a),$$
is an algebraic integer, i.e., belongs to $\bar{\Z} \subset \bar{\Q}$.
\end{lemma}
\begin{proof}Observe that $B_{1,\chi} = \frac{1}{mn}\sum_{a=1}^{mn} a \chi(a),$ 
has a possible fraction by $mn$, and that 
in this sum over $a \in \{1,2,\cdots, mn\}$, if we instead sum over an arbitrary set $A$ of 
integers which have these residues mod $mn$, then $\frac{1}{mn}\sum_{a \in A} a \chi(a),$ 
will differ from $B_{1,\chi}$ by an
integral element (in $\bar{\Z}$). Since our aim is to prove that $B_{1,\chi}$ is integral,  it suffices to prove that 
$\frac{1}{mn}\sum_{a \in A} a \chi(a)$ is integral for some set of representatives $A \subset \Z$ of residues mod $mn$.

For an integer $a \in \{1,2,\cdots, m\}$, let $\bar{a}$ be an arbitrary  integer whose reduction mod $m$ is $a$, and 
whose reduction mod $n$ is 1. Similarly, for an integer $b \in \{1,2,\cdots, n\}$, let $\bar{b}$ be an arbitrary  integer whose reduction mod $n$ is $b$ and whose reduction mod $m$ is 1. Clearly, the set of integers $\bar{a} \cdot \bar{b}$ 
represents --- exactly once --- each residue class mod $mn$, and that $\bar{a} \cdot \bar{b}$ as an element in $\Z$ 
goes to the pair $(a,b) \in \Z/m \times \Z/n$. 
(It is important to note that $\bar{a} \cdot \bar{b}$ as an element in $\Z$ is {\it not} congruent to ${ab}$ mod $mn$, and therein lies a subtlety in the Chinese remainder theorem: there is no simple inverse to the natural isomorphism: $\Z/mn \rightarrow  \Z/m \times \Z/n$.)

By definition of the character $\chi$, $\chi(\bar{a} \cdot \bar{b}) = \chi_1(a) \chi_2(b)$. It follows that,
 $$ \frac{1}{mn}\sum \bar{a}\bar{b} \chi(\bar{a}\cdot \bar{b}) 
-  \left [\frac{1}{m}\sum_{a=1}^{m} a \chi_1(a) \right ] \cdot 
\left [\frac{1}{n}\sum_{b=1}^{n} b \chi_2(b) \right ] \in \bar{\Z}. \quad \quad (\star)$$

Since the character $\chi$ is odd, one of the characters, say $\chi_2$ is even (and $\chi_1$ is odd).

Observe that, $$B_{1,\chi_2} = \frac{1}{n}\sum_{b=1}^{n} b \chi_2(b) = \frac{1}{n}\sum_{b=1}^{n} (n-b) \chi_2(b)  .$$ 
 It follows that, 
$$\frac{2}{n}\sum_{b=1}^{n} b \chi_2(b) = \sum_{b=1}^{n}  \chi_2(b) =0  ,$$ 
where the last sum is zero because the character $\chi_2$ is assumed to be non-trivial.

Since  $$ \frac{1}{mn}\sum \bar{a}\bar{b} \chi(\bar{a}\cdot \bar{b}) - \frac{1}{mn}\sum_{c=1}^{mn} c \chi(c) \in \bar{\Z},$$ 
by the equation $(\star)$, it follows that:
 $$ \frac{1}{mn}\sum_{c=1}^{mn} c \chi(c) \in \bar{\Z},$$
as desired.\end{proof}

\begin{lemma}
For $p$ a prime, 
let $\chi $ be a primitive Dirichlet character on $(\Z/p^n\Z)^\times$
with $\chi(-1)=-1$.  
Write $(\Z/p^n\Z)^\times = (\Z/p\Z)^\times \times  (1+p\Z/1+p^n\Z)$, and
the character $\chi $ as $\chi_1 \times \chi_2$ with respect to this decomposition. 
Then,
$$L(0,\chi) = -B_{1,\chi} = -\frac{1}{p^n}\sum_{a=1}^{p^n} a \chi(a),$$
is an algebraic integer, i.e., belongs to $\bar{\Z} \subset \bar{\Q}$ if and only if $\chi_1 \not = \omega_p^{-1}$.
\end{lemma}
\begin{proof}
Assuming that $\chi_1 \not = \omega_p^{-1}$, we prove that $B_{1,\chi}$
belongs to $\bar{\Z} \subset \bar{\Q}$.

By an argument similar to the one used  in the previous lemma, it can be checked that,
 $$ \frac{1}{p^n}\sum_{a=1}^{p^n} {a} \chi({a}) 
-  \left [\frac{1}{p}\sum_{a=1}^{p} a \chi_1(a) \right ] \cdot 
\left [\frac{1}{p^{n-1}}\sum_{b=1}^{p^{n-1} } (1+bp) \chi_2(1+bp) \right ] \in \bar{\Z}. \quad \quad (\star).$$

If $\chi_1 \not = \omega_p^{-1}$, $\frac{1}{p}\sum_{a=1}^{p} a \chi_1(a)$ is easily seen to be integral. To prove the lemma, 
it then suffices to
prove that, $\left [\frac{1}{p^{n-1}}\sum_{b=1}^{p^{n-1}} (1+bp) \chi_2(1+bp) \right ] $ is integral.

Note the isomorphism of the additive group $\Z_p$ with the multiplicative group $1+p\Z_p$ by the map
$n\rightarrow (1+p)^n \in 1 + p\Z_p$. Let $\chi_2(1+p) = \alpha$ with $\alpha^{p^{n-1}}=1$. 

Then (the first and third equality below is up to $\bar{\Z}$),
\begin{eqnarray*}
\frac{1}{p^{n-1}}\sum_{b=1}^{p^{n-1} } (1+bp) \chi_2(1+bp) 
 & = & \frac{1}{p^{n-1}}\sum_{c=1}^{p^{n-1}} (1+p)^c \alpha^c  \\
& = & \frac{1}{p^{n-1}} 
\sum_{c=1}^{p^{n-1}} [\alpha(1+p)]^c  \\
& = & \frac{1}{p^{n-1}}  \frac{1 - [\alpha(1+p)]^{p^{n-1}}}{1 - \alpha(1+p)} \\
& = & \frac{1}{p^{n-1}}  \frac{[1 - (1+p)^{p^{n-1}}]}{[1 - \alpha(1+p)]}.
\end{eqnarray*}

Note that since $\alpha^{p^{n-1}}=1$ either $\alpha = 1$, or $1-\alpha$ 
is a uniformizer in $\Q_p(\zeta_{p^d})$ 
for some $d \leq n-1$. Therefore either $ -p = [1-\alpha(1+p)]$ if $\alpha =1$,
or $[1-\alpha(1+p)]$ is a uniformizer in $\Q_p(\zeta_{p^d})$. Finally, it suffices to observe that,

$$(1+p)^{p^{n-1}} \equiv 1 \bmod p^{n},$$
hence, $\frac{1}{p^{n-1}}\sum_{b=1}^{p^{n-1}} (1+bp) \chi_2(1+bp) $ is integral.

If $\chi_1 = \omega_p^{-1}$, the same argument gives non-integrality; we omit the details.
\end{proof}

The following proposition follows by putting the previous two lemmas together, and making an argument similar to what 
went into the proof of these two lemmas. We omit the details. 
\begin{proposition} \label{prop1}
Primitive Dirichlet characters $\chi: (\Z/n)^\times \rightarrow \bar{\Z}_p^\times$ for which 
$L(0,\chi)$ does not belong to $\bar{\Z}_p$ are exactly those for which:

\begin{enumerate}
\item $n = p^d$.

\item $\chi = \omega_p^{-1}$ mod $\p$.
\end{enumerate}
 
\end{proposition}

The following consequence of the proposition suggests that prudence is to be exercised 
when discussing congruences of $L$-values
for Artin representations which are congruent.

\begin{cor}Let $p,q$ be odd primes with $p|(q-1)$. For any character $\chi_2$ of $(\Z/q\Z)^\times$ of order $p$, 
define the character 
$ \chi = \omega_p^{-1} \times \chi_2$ of $(\Z/pq\Z)^\times$. Then although the characters 
$\omega_p^{-1}$ and $\chi$ have the same reduction modulo $p$, 
$L(0,\omega_p^{-1})$ is $p$-adically non-integral whereas $L(0,\chi)$ is integral.
\end{cor}

\begin{question}
Let $\chi: (\Z/p^dm)^\times \rightarrow \bar{\Z}_p^\times$ with $(p,m)=1$, $m> 1$,
be a primitive Dirichlet character
 for which $\chi = \omega_p^{-1}$ mod $\p$ so that by Proposition \ref{prop1}, $L(0,\chi)$ is $\p$-integral. Is it possible to
have $L(0,\chi)=0$  modulo $\p$, the maximal ideal of $\bar{\Z}_p$? Our proofs in this section are `up to $\bar{\Z}$', so good to detect integrality, but not good for questions modulo $\p$. The question is relevant to conjecture \ref{classgroups} 
to see if the character $\omega_p$ appears in the classgroup  $H/H^+$ for $E= \Q(\zeta_{p^dm})$; 
such a character is known not to appear in the classgroup of   $H/H^+$ for $E= \Q(\zeta_{p^d})$. 
\end{question}

\section{Congruences   and their failure for $L$-values}

This paper considers integrality properties of certain Artin $L$-functions at 0. It may seem most natural 
 that if two such Artin representations $\rho_1, \rho_2: \Gal(\bar{\Q}/F) \rightarrow \GL_n(\bar{\Q}_p)$
 have the same semi-simplification mod-$p$ and do not contain the character $\omega_p^{-1}$, then
 $L(0,\rho_1)$ and $L(0,\rho_2)$
 which are in $\bar{\Z}_p$ by Conjecture 1,  have the same reduction mod-$p$.
 This is not true even in the simplest  case of Dirichlet characters for $\Q$. It is possible
 to fix this problem for abelian characters of $\Q$, and more generally for any totally real number field which is what this section
 strives to do, cf. Proposition \ref{general prop}. The recipe given in Proposition \ref{general prop} immediately suggests itself in the non-abelian case, but we have not spelt it out. 

 The problem that we find
 dealing with abelian characters $\chi_1,\chi_2$ is that they may be congruent for some prime, but may have
 different conductors in which case it is not the $L$-values $L(0,\chi_1)$ and $L(0,\chi_2)$ which are congruent, but a modified
 $L$-value, say $L_f(0,\chi)$  which gives the right congruence; these $L$-values are product of
 $\prod_\wp  (1-\chi(\wp))$ with $L(0,\chi)$
 where $\wp$ are all primes dividing either the conductor of $\chi_1$ or of $\chi_2$.  
 It appears to this author that this recipe of modifying the $L$-function, which amounts to
 dropping some Euler factors,  is not quite there in the literature (which
 is mostly about dropping the Euler factors above $p$).

 We begin with some elementary lemmas which go into congruences of $L$-values at 0 of Dirichlet characters, and
 then we consider totally real number fields. Most  of this section is  the outcome of discussions with
 G. Asvin  whom I thank heartily.

 \begin{lemma} \label{sim}   Let $V$ be a vector space over $\Q$, and $\chi = \chi_f: \Z/f \rightarrow V$
   be any function with the property that
  $\sum_{a=1}^f \chi(a)=0$. Let $\chi_{df}$ be the function on $\Z/df$ obtained
  from $\chi$ by composing with the natural map $\Z/df\rightarrow \Z/f$. Then
\begin{enumerate}
\item   $L(0, \chi)
 :=\frac{1}{f} \sum_{a=1}^{f}  a \chi(a)
  = \frac{1}{df} \sum_{a=1}^{df}  a \chi_{df}(a) := L(0, \chi_{df}).$
\item Let $\chi: (\Z/f)^\times \rightarrow \C^\times$ be a primitive character of conductor $f$ with  $\chi \not = 1$. Then for
  any $f|f'$, 
  $$ \frac{1}{f'} \sum_{ a=1,  (a,f')=1 }^{f'}  a \chi(a) =
  \prod_{p|f'}(1-\chi(p)) L(0, \chi).$$
\end{enumerate}
\end{lemma}
\begin{proof} Observe that,
  \begin{eqnarray*} \sum_{a=1}^{df}  a \chi_{df}(a)
    & = & \sum _{a=1}^{f} \sum_{i=0}^{d-1} (a+if) \chi_{df}(a+if) \\
    & = & \sum _{a=1}^{f} \sum_{i=0}^{d-1} (a+if) \chi(a) \\
    & = & \sum _{a=1}^{f}  \left [d a \chi(a) +\frac{fd(d-1)}{2} \chi(a) \right ] \\
      & = & d \sum _{a=1}^{f}a \chi(a),
  \end{eqnarray*}
  where in the last step we have used that $\sum_{a=1}^f \chi(a)=0$.
  The proof of part $(1)$ of the lemma follows.

  The proof of part $(2)$ will proceed in several steps, according to the value of $f'$. Observe first that
  if $f$ and $f'$ have the same prime divisors,
  then $(a,f)=1$ if and only $(a,f')=1$. Therefore,
  $$ \frac{1}{f'} \sum_{ a=1,  (a,f')=1 }^{f'}  a \chi(a) =
  \frac{1}{f'} \sum_{ a=1}^{f'}  a \chi_{f'}(a)
  = \frac{1}{f} \sum_{ a=1}^{f}  a \chi_{f}(a),$$
  where the second equality is a consequence of part (1) of the Lemma.
  In this case, i.e., when  $f$ and $f'$ have the same prime divisors, for all $p|f'$, $\chi(p)=0$. It follows
  that
$\prod_{p|f'}(1-\chi(p)) =1$, proving this case of part (2) of the Lemma.

Assume next that, $f'=fp^m$,
 $m \geq 1$, $p$ a prime with $(f,p)=1$.
 In this case    note that (using the notation on $L(0, \chi_{fp^m})$ introduced in part (1) of the Lemma),
 \begin{eqnarray*} L(0, \chi_{fp^m})
   & = & \frac{1}{fp^m} \sum_{ a=1,  (a,fp^m)=1 }^{fp^m}  a \chi(a) 
    + \frac{1}{fp^m} \sum_{i=1}^{fp^{m-1}}  (pi) \chi_{fp^{m-1}}(pi) \\
    & = &   \frac{1}{fp^m} \sum_{ a=1,  (a,fp^m)=1 }^{fp^m}  a \chi(a) + \chi(p) L(0, \chi_{fp^{m-1}}).
  \end{eqnarray*}

    Since by part (1) of the lemma, $L(0, \chi_{fp^m}) = L(0, \chi_{fp^{m-1}}),$ 
    the proof of part $(2)$ follows in this case.

    For general $f' = df''$, with $f''$ having the same prime divisors as $f$, and $d$ having prime divisors
which are coprime to those of $f$, let $    d=p_1^{m_1} \cdots p_r^{m_r}$.
    We argue by induction on $r$, thus assuming the result
    for $d_{r-1}= p_1^{m_1} \cdots p_{r-1}^{m_{r-1}}$,
    adding the prime power $p_r^{m_r}$ at the end which proves (2) for $d=p_1^{m_1} \cdots p_r^{m_r}$ using
    part (1), and part of (2) just proved for prime powers (to be used for $p_r^{m_r}$).
  \end{proof}

\begin{lemma} \label{cong}
  Let $\chi_1, \chi_2: (\Z/f)^\times \rightarrow \bar{\Z}_p^\times$
  be two (not necessarily primitive) characters. Consider $\chi_1,\chi_2$ as functions on $\Z/f$
  by declaring their values outside of $(\Z/f)^\times$ to be zero.
  Assume that the reduction mod $\wp$, 
  $\bar{\chi}_1, \bar{\chi}_2: (\Z/f)^\times \rightarrow \bar{\F}_p^\times$ are the same.
  If $p|f$, assume that neither of the $\bar{\chi}_1, \bar{\chi}_2: (\Z/f)^\times \rightarrow \bar{\F}_p^\times$
  factors through  $(\Z/p)^\times \rightarrow \bar{\F}_p^\times$
  to give $\omega_p^{-1}$
  where $\omega_p$ is the natural map from $(\Z/p)^\times$  to ${\F}_p^\times$.
    Then,
    $L_f(0, \chi_1)  := \frac{1}{f} \sum_{a=1}^{f}  a \chi_1(a)$
    and  $L_f(0, \chi_2)
  := \frac{1}{f} \sum_{a=1}^{f}  a \chi_2(a)$ are   in $\bar{\Z}_p$, and have the same reduction to  $ \bar{\F}_p$.
  \end{lemma}

\begin{proof}
By the hypothesis in the lemma,
there is a   $b \in (\Z/f)^\times$ such that  $[b\chi_1(b)-1] \in \bar{\Z}_p^\times$, and hence also
  $[b\chi_2(b)-1] \in \bar{\Z}_p^\times$. Fix such a $b \in (\Z/f)^\times$.

For $a \in \{1,\cdots, f\}$, write $$ab = [ab] + \lambda_a f,$$
with $[ab] \in \{1,\cdots, f\}$.

From the definition of $L_f(0,\chi_1)$,
\begin{eqnarray*}
  b\chi_1(b)L_f(0,\chi_1) & = & \frac{1}{f}\sum_{a=1}^{f} ab\chi_1(ab)
\end{eqnarray*}
As $b$ is invertible in $\Z/f$, $a\rightarrow [ab]$ is a bijection on $\{1,\cdots, f\}$,
therefore the above equation yields,
\begin{eqnarray*}
  [b\chi_1(b)-1]L_f(0,\chi_1) & = & \sum_{a=1}^{f} \lambda_a\chi_1(ab) \quad \quad \quad  \hfill(\star1).
\end{eqnarray*}
Similarly, 
\begin{eqnarray*}
  [b\chi_2(b)-1]L_f(0,\chi_2) & = & \sum_{a=1}^{a=f} \lambda_a\chi_2(ab)\quad \quad \quad \hfill(\star2).
\end{eqnarray*}
Since $\chi_1,\chi_2$ are congruent, the right hand sides of $(\star1), (\star2)$ are the same in $ \bar{\F}_p$,
and by the choice of $b$ made in the beginning of the proof of the lemma, $[b\chi_1(b)-1]$ as well as $[b\chi_2(b)-1]$ are 
in $ \bar{\F}_p^\times$, and are the same, thus it follows that $L_f(0,\chi_1)$ and $L_f(0,\chi_2)$ are in
$ \bar{\Z}_p$,
and are the same in $ \bar{\F}_p$.
  \end{proof}

\begin{proposition}
  Let $f_1, f_2$ be integers, and $f$ any integer divisible by both $f_1,f_2$. Suppose $\chi_1$ (resp. $\chi_2$) are primitive Dirichlet characters
  of conductors $f_1$ (resp. $f_2$) with values in $\bar{\Z}_p^\times$.
If $p|f$, assume that $\bar{\chi}_1, \bar{\chi}_2: (\Z/f)^\times \rightarrow \bar{\F}_p^\times$
  do  not factor through  $(\Z/p)^\times \rightarrow \bar{\F}_p^\times$
  to give $\omega_p^{-1}$
  where $\omega_p$ is the natural map from $(\Z/p)^\times$ to $ {\F}_p^\times$.
  Then $L_f(0, \chi_1)  := \frac{1}{f} \sum_{a=1}^{f}  a \chi_1(a)$ (where $\chi_1$ is considered as a function on $\Z/f$ zero outside
  $(\Z/f)^\times$)   has the value given by:
  $$L_f(0, \chi_1) = \prod_{p|f}  (1-\chi_1(p)) \cdot L(0,\chi_1);$$
  similarly for  $L_f(0, \chi_2)$. Both $L(0,\chi_1)$ and $L(0,\chi_2)$ are 
$\bar{\Z}_p$, and   if $\chi_1$ and $\chi_2$ are congruent modulo the maximal ideal in $\bar{\Z}_p$,
    so is the case for $L_f(0,\chi_1)$ and $L_f(0,\chi_2)$.
  \end{proposition}

\begin{proof} $L_f(0, \chi_1)  := \frac{1}{f} \sum_{a=1}^{f}  a \chi_1(a)$ has the value as asserted in the proposition
  by part (2) of Lemma \ref{sim}, and their
  congruence holds by Lemma \ref{cong}.
  \end{proof}
\begin{cor}
  If $\chi_1$ and $\chi_2$ have conductors $f_1$ and $f_2$ such that the prime divisors of $f_1$ and $f_2$ are the same, then
  for $f$ which is the least common multiple of $f_1$ and $f_2$, $L_f(0,\chi_1) = L(0, \chi_1)$ and $L_f(0,\chi_2) = L(0,\chi_2)$,
  hence if $\chi_1$ and $\chi_2$ are congruent, so are $L(0,\chi_1)$ and $L(0,\chi_2)$. On the other hand, suppose $f_1=p$, $f_2=pq$,
  $\chi_2= \chi_1 \times \alpha: (\Z/p)^\times \times (\Z/q)^\times = (\Z/pq)^\times \rightarrow \bar{\Z}_p^\times$ with $\bar{\alpha}=1$,
    then,
    \begin{eqnarray*} L_{pq}(0, \chi_2) & = & L(0,\chi_2) \\
      L_{pq}(0, \chi_1)
      & = & (1-\chi_1(q))L(0,\chi_1).
    \end{eqnarray*}
    Since $L_{pq}(0, \chi_1)$ and $L_{pq}(0, \chi_2)$ are congruent mod $\wp$,  we find that if $L(0, \chi_1)$ and $L(0, \chi_2)$
    are not both zero  mod $\wp$, they cannot be the same mod $\wp$ since $(1-\chi_1(q))$ cannot be 1 mod $\wp$.  
  \end{cor}

\begin{remark}
  The hypothesis that the reduction mod $\wp$ of $\chi: (\Z/f)^\times \rightarrow \bar{\Z}_p^\times$ does not factor through
  $(\Z/p)^\times \rightarrow \bar{\F}_p^\times$ to give $\omega_p^{-1}$ is stonger than what is required for $p$-integrality
  of $L(0,\chi)$. For example, by Lemma \ref{simple}, $L(0,\chi)$ is integral if the conductor of $\chi$ has two distinct prime factors,
  say $f=pq$, with $(\Z/f)^\times = (\Z/p)^\times \times (\Z/q)^\times$, and $\chi = \alpha \times \beta$. In Lemma \ref{cong} we would
  excluding characters $\chi$ for which $\bar{\alpha} = \omega_p^{-1}$, and $\bar{\beta} = 1$. 
\end{remark}

Here is the more general version of the previous proposition.

\begin{proposition} \label{general prop}Let $F$ be a totally real number fields with $G = \Gal(\bar{\Q}/F)$ its Galois group. For
  $\chi: G\rightarrow \bar{\Z}_p^\times$,  a character of finite order of $G$, which is also to be
  considered as a character o the groups of ideals coprime to a nonzero ideal $\wf$ in $K$ by classfield theory (so $\wf$ is divisible by the conductor
  of $\chi$, but may not be the conductor of $\chi$).
  Let $L(s,\chi)$ be the `true' $L$-function associated to the character $\chi$, and define $L_{\wf}(s,\chi)$ by
  $$L_{\wf}(s,\chi) = \sum _{ (\wa,\wf)=1} \frac{\chi(\wa)}{(N\wa)^s},$$
  where $N\wa$ denotes the norm of an integral ideal $\wa$ in $F$. Then,

  \begin{enumerate} \item $L_{\wf}(s, \chi) = \prod_{\wp|\wf}  (1-\frac{\chi(\wp)}{(N\wp)^{s}}) \cdot L(s,\chi_1).$

\item For any integral
  ideals $\wc$ in $F$ coprime to $p\wf$, integers $k\geq 1$, 
  $$\Delta_{\wc}(1-k, \chi) = (1-\chi(\wc)N\wc^k)L_\wf(1-k,\chi),$$
  are in $\bar{\Z}_p$.

  \item If $\chi_1$ and $\chi_2$ are two characters of $G$ with values in $\bar{\Z}_p^\times$
  with conductors dividing $\wf$, 
  such that neither of the two  reductions  $\bar{\chi}_1,\bar{\chi}_2: G \rightarrow \bar{\F}_p^\times$ is $\omega_p^{-1}$,
  then $L_{\wf}(0,\chi_1)$ and $L_{\wf}(0,\chi_2)$ are in $\bar{\Z}_p$,
and   if $\chi_1$ and $\chi_2$   are congruent modulo the maximal ideal in $\bar{\Z}_p$,
  so is the case for $L_{\wf}(0,\chi_1)$ and $L_{\wf}(0,\chi_2)$.
  \end{enumerate}
  \end{proposition}
\begin{proof}
    Deligne and Ribet, cf. Theorem 0.4 of \cite{DR}, see also Theorem 2.1 of
  \cite{DR2}, and Proposition 1.4 of \cite{DR2}, prove integrality of $\Delta_{\wc}(1-k, \chi)$,
  as well as the congruence between $\Delta_{\wc}(1-k, \chi_1)$ and $\Delta_{\wc}(1-k, \chi_2)$. 

  The final congruence between  $L_{\wf}(0,\chi_1)$ and $L_{\wf}(0,\chi_2)$ follows as in Lemma \ref{cong}
  by choosing an integral ideal $\wc$ coprime to $p\wf$ such that $(1-\chi(\wc)N\wc)$ is a unit in $\bar{\Z}_p$
  which follows just as in Lemma \ref{cong} for $\chi=\chi_1$ (hence for $\chi=\chi_2$ too),
  because $\bar{\chi}_1: G \rightarrow \bar{\F}_p^\times$ is not $\omega_p^{-1}$ (see beginning of section 2 of \cite{DR2}
  how the `norm map'  under the identification of characters on ideals coprime to $\wf$ to characters of $G$ becomes action
  of $G$ on $p$-power roots of identity).
\end{proof}

\vspace{ 1 cm}
\noindent{\bf Acknowledgement:} The 
author thanks P. Colmez for suggesting that the questions posed here are not as outrageous as one might think, 
and for even suggesting that some of the conjectures above should have an affirmative answer as a consequence 
of the Main conjecture of Iwasawa theory for totally real number fields 
proved by A. Wiles \cite{Wiles} if one knew the vanishing of the $\mu$-invariant (which is a 
conjecture of Iwasawa proved for abelian extensions of $\Q$ by Ferrero-Washington). I also thank U.K. Anandavardhanan, C. Dalawat, C. Khare for their comments and their encouragement. 
I am particularly grateful to G. Asvin  for his help with the last section.

The final writing of this work was supported by a grant of
the Government of the Russian Federation
for the state support of scientific research carried out
under the supervision of leading scientists,
agreement 14.W03.31.0030 dated 15.02.2018.

\vspace{1cm}

\noindent
{Tata Insititute of Fundamental Research,
Colaba,
Mumbai-400005, India.}

\noindent{Laboratory of Modern Algebra and Applications,
Saint-Petersburg State University,
Russia}

\noindent{Email: prasad.dipendra@gmail.com}

\end{document}